\documentclass{article}
\usepackage{amsmath}
\usepackage{amssymb}
\usepackage[all]{xy}
\usepackage{enumitem}
\usepackage{amsthm}
\usepackage[utf8]{inputenc}
\usepackage[english]{babel}
\usepackage{fancyhdr}
\usepackage{graphicx}
\newcommand{\N}{\ensuremath{\mathbb{N}}}
\newcommand{\Z}{\ensuremath{\mathbb{Z}}}
\newcommand{\R}{\ensuremath{\mathbb{R}}}

\newcommand{\C}{\ensuremath{\mathbb{C}}}

\newcommand*{\defeq}{\mathrel{\vcenter{\baselineskip0.5ex \lineskiplimit0pt
			\hbox{\scriptsize.}\hbox{\scriptsize.}}}%
	=}

\theoremstyle{plain}
\newtheorem{theorem}{Theorem}[section]
\newtheorem{prop}[theorem]{Proposition}
\newtheorem{defin}[theorem]{Definition}
\newtheorem{lemma}[theorem]{Lemma}

\newtheorem{rem}[theorem]{Remark}

\newtheorem{example}[theorem]{Example}
\renewcommand{\thetheorem}{\arabic{section}.\arabic{theorem}}

\title{The motivic Adams and Adams-Novikov spectral sequences at odd primes}
\author{Sven-Torben Stahn}
\begin{document}

\maketitle
\begin{abstract}
We study the motivic Adams-Novikov spectral sequence at an odd prime over the complex and real numbers to compute the stable motivic homotopy groups of a suitable completion of the motivic sphere spectrum. We show that for odd primes over our choice of base fields the classical Adams-Novikov spectral sequence determines the motivic one and its differentials.
\end{abstract}	

\section{Introduction}
The motivic Adams spectral sequence is the motivic counterpart of the well known classical computational device. It was introduced by Morel and then further studied by Dugger-Isaksen (for example in \cite{DI} and \cite{Isa}), Hu-Kriz-Ormsby \cite{HKO} and others. Dugger and Isaksen have used it for extensive computations over the base fields \(\C\) at the prime \(2\) (see e.g. \cite{DI}) and over the real numbers \(\R\). They also used the additional information available in the motivic spectral sequence to deduce new information about the classical Adams spectral sequence \cite[Chapter 8]{DI}.
In this article we use the motivic Adams spectral sequence and the motivic Adams-Novikov spectral sequence to study the stable motivic homotopy groups of the spheres at odd primes. In Section 2 and 3 we recollect some basic facts about the motivic Adams and Adams-Novikov spectral sequences, along with some other preparational material. In section 4, we study the motivic versions of the Cartan-Eilenberg and the algebraic Novikov spectral sequence and their relationship. In the final section we show that for odd primes over the complex (Proposition \ref{PropChapter2}) and real (Proposition \ref{PropChapter2addendum}) numbers the classical Adams-Novikov spectral sequence determines the motivic one and its differentials. In a certain range, we can completely exclude the possibility of \(\tau\)-primary torsion (Proposition \ref{tauIsoProp}).
Apart from \(\tau\)-primary torsion (resp. $\theta$-primary torsion in the real case) associated to the existence of nontrivial differentials in the topological Adams-Novikov spectral sequence, the motivic Adams-Novikov spectral sequence looks very similar to the classical one. For a concrete example, we refer to \ref{firstExample}.\\

\subsection*{Acknowledgements}
This research was originally a part of my dissertation and conducted in the framework of the research training group
\emph{GRK 2240: Algebro-Geometric Methods in Algebra, Arithmetic and Topology},
which is funded by the DFG. I am indebted to my advisor Jens Hornbostel and Sean Tilson for helpful discussions.\\

\section{The motivic stable homotopy category}
In this paper we work in the motivic stable homotopy category \(\mathcal{SH}_k\) of \(\mathbb{P}^1\)-spectra over the base field \(k\) of Morel and Voevodsky (\cite{VOE} and \cite{MV}), where \(k\) is either the real or complex numbers. We use the following grading convention for the motivic spheres:
\begin{defin}
Define \(S^{1,0}\) as the \(\mathbb{P}^1\)-suspension spectrum of the simplicial sphere \((S^1,1)\) and \(S^{1,1}\) as the \(\mathbb{P}^1\)-suspension spectrum of \((\mathbb{A}^1-0,1)\). The suspension spectrum of \(\mathbb{P}^1\) is then equivalent to \(S^{2,1}\). Define 
\[S^{p,q}\defeq(S^{1,0})^{\wedge(p-q)}\wedge (S^{1,1})^q.\]
\end{defin}

This relates to the other common notation of \(S^{\alpha} = S^{1,1}\) by \(S^{p,q}=S^{p-q+q\alpha}\).\\
The spheres define the following subcategory of cellular motivic spectra (see \cite[Definition 8.1]{DI2}):
\begin{defin}

The category of finite cellular spectra \(\mathcal{SH}_k^{cell}\) in \(\mathcal{SH}_k\) is defined (c.f. \cite[Definition 2.1]{DI2}) as the smallest full subcategory that contains the spheres \(S^{p,q}\) and is closed under weak equivalences and cofiber sequences. The category of cellular spectra is the smallest such subcategory that is also closed under arbitrary colimits.
\end{defin}

By \cite[Lemma 2.2]{ROEN} a motivic spectrum is cellular if and only if it admits a cell presentation, i.e. it can be built by successively attaching cells \(S^{s,t}\).
\begin{defin}
\label{finiteType}
 A motivic cell spectrum \(X\) is called \emph{of finite type} if there is a \(k\in \N \) such that \(X\) admits a cell presentation with no cells \(S^{s,t}\) for \(s-t<k\) and only finitely many cells in other bidegrees. See \cite[Section 2]{HKO}.
\end{defin}

Over the base fields \C\ there is a topological realization functor called Betti realization \[R: \mathcal{SH}_\C \rightarrow \mathcal{SH}_{top}\] mapping into the stable homotopy category. Over the base field \R\ there is a similar functor \[R': \mathcal{SH}_\R \rightarrow \mathcal{SH}_{C_2}\]
mapping into the stable \(C_2\)-equivariant homotopy category. Both functors are strict symmetric monoidal left Quillen functors. There are many references for the construction and basic properties of these functors. We will use \cite[4.3]{JOA} as our main reference.
Betti realization maps the suspension spectrum of a smooth scheme over \C\ to the suspension spectrum of the topological space of its complex points, endowed with the analytic topology. The real realization functor maps the suspension spectrum of a smooth scheme over \R\ to the space of its complex points, endowed with the analytic topology, and the \(C_2\)-action is provided by complex conjugation. Consequently the behaviour of these functors on the sphere spectra is given by the assignments \(S^{p,q}\overset{R}\mapsto\ S^p\) and \(S^{p,q}\overset{R'}\mapsto S^{p-q+q\sigma}\), where \(\sigma\) denotes the sign representation of \(C_2\).\\
Because Betti realization maps the motivic spheres to the topological ones, it induces maps on homotopy groups
\[R: \pi_{pq}(X)\rightarrow \pi_p(R(X))\]
for every motivic spectrum \(X\in\mathcal{SH}_\C\). Therefore for every motivic spectrum \(E\in\mathcal{SH}_\C\) it also induces maps \[R: E_{pq}(X)\rightarrow R(E)_p(R(X))\]
and
\[R: E^{pq}(X)\rightarrow R(E)^p(R(X))\] on homology and cohomology  associated to that spectrum.\\
\ \\
The functors \(R\) and \(R'\) each have a right inverse
\[
c: \mathcal{SH}_{top} \rightarrow \mathcal{SH}_\C
\]
and
\[
c': \mathcal{SH}_{C_2} \rightarrow \mathcal{SH}_\R 
\]
which are called the constant simplicial presheaf functors. Both are strict symmetric monoidal. It is a result of Levine (see \cite[Theorem 1]{LEV}) that \(c\) is not only faithful but also full. For \(c'\) this only holds after a suitable completion (see \cite[Theorem 1]{HO}).\\
\ \\
The functor 
\(p^*:\mathcal{SH}_\R \rightarrow \mathcal{SH}_\C\)
induced by the map of schemes \(p:\text{Spec}(\C)\rightarrow \text{Spec}(\R)\)
relates the different realization functors in the following way (c.f. \cite[Proposition 4.13]{HO}): 
\[
\xymatrix{
\mathcal{SH}_\R \ar[r]^{p^*}\ar[d]^{R'} &
\mathcal{SH}_\C \ar[d]^R\\
\mathcal{SH}_{C_2} \ar[r]^{res^{C_2}_e}&
\mathcal{SH}_{top}\\
}\]
Here \(res^{C_2}_e\) denotes the functor which restricts from \(C_2\) to the trivial group.

\subsection{Coefficients of motivic cohomology and the dual motivic Steenrod algebra}
One key ingredient for the Adams spectral sequence is knowlegde of the Steenrod algebra or of the dual Steenrod algebra. Motivically, the Steenrod Algebra was described by Voevodsky for fields of characteristic zero and later by Hoyois, Kelly and {\O}stv{\ae}r in positive characteristic. While some interesting phenomena occur at the prime two, the motivic Steenrod algebra is more closely related to the classical topological Steenrod algebra at odd primes. To describe the motivic Steenrod algebra it is sufficient to know the coefficients of motivic coholomogy with \(\Z/l\Z\) - coefficients, so we will compute these coefficients over the base fields \(\R\) and \(\C\):

\begin{prop}Let \(l\neq 2\). be a prime\\
\begin{enumerate}
\item The coefficients \(H\Z/l^{**}\) of motivic cohomology are given as a ring by \[H\Z/l^{**}\cong \Z/l [\tau] \] with \(|\tau|=(0,1)\)  over \(k=\C\).
\item The coefficients \(H\Z/l^{**}\) of motivic cohomology are given as a ring by \[H\Z/l^{**}\cong \Z/l [\theta] \] with \(|\theta|=(0,2)\)  over \(k=\R\).
\item The map \(H\Z/l_\R^{**}\rightarrow H\Z/l_\C^{**}\) is given by \(\theta \mapsto \tau^2\).
\item For a fixed bidegree \(p,q\) with \(p\leq q\) there is a commutative square
\[\xymatrix{
 H_\R^{p,q}(pt,\Z/l)\ar[d]^{R'}_\cong \ar[r] & H_\C^{p,q}(pt,\Z/l)\ar[d]^{R}_\cong\\
 H^{p-q+q\sigma}(pt,\underline{\Z/l})\ar[r]^{res^{C_2}_e} & H_{sing}^p(pt, \Z/l)\\
}\]
\\Here \(H_k\) denotes motivic cohomology over the base field \(k\), \(H^{p-q+q\sigma}(pt,\underline{\Z/l})\) denotes Bredon cohomology (graded by the trivial representation and the sign representation \(\sigma\)) and \(H_{sing}\) denotes singular cohomology. The top map is the one induced by \(\text{Spec}(\C)\rightarrow \text{Spec}(\R)\), the bottom map is the restriction functor to the trivial group.\\
In particular, both \(\tau\) and \(\theta\) do not vanish under topological realization.
\end{enumerate}
\end{prop}
\begin{proof}
We know that \(H\Z/l^{**}=0\) for $q<p$ ((cf. \cite[Theorem 3.6]{MVW})).\\ Let \(q\geq p\). Then there is an isomorphism from motivic to \'etale cohomology: \[H^{p,q}(Spec(k), \mathbb{Z}/l) \cong H_{\acute{e}t}^p(k,\mu_l^{\otimes q})\]
This isomorphism respects the product structure (\cite[1.2,4.7]{GL}).\\
The \'etale cohomology groups \(H_{\acute{e}t}^p(k,\mu_l^{\otimes q})\) can be computed as the Galois cohomology of the separable closure of the base field (in both cases the complex numbers) with coefficients in the $l$-th roots of unity. The action of the absolute Galois group $G$ is given by the trivial action if $k=\mathbb{C}$ and by complex conjugation if $k=\mathbb{R}$:
\[H_{\acute{e}t}^p(k,\mu_l^{\otimes q})\cong H^p(G, \mu_l^{\otimes q}(\C))\]

\begin{enumerate}
\item For $k=\mathbb{C}$, these groups all vanish for $p\neq 0$ because the Galois group is the trivial group. In the degree $p=0$ they are isomorphic to $\mathbb{Z}/l$ for all $q\geq 0$. The multiplicative structure is given by the tensor product of the modules.\\
\item For $k=\mathbb{R}$, we have the following isomorphism of \(G\)-modules:\\
\xymatrix{
 \mu_l \otimes \mu_l \ar[d]^\cong \ar[rr] & & \mu_l \ar[d]_\cong\\
 \Z/l \otimes \Z/l\ar[rr]_{a\otimes b \mapsto ab}^\cong & & \Z/l\\
}\\
Here $G$ acts on the top left hand side by complex conjugation on each factor, on the lower left hand side by the assignment \(x\mapsto -x\) on each factor, and trivial on the right hand side. Hence $\mu_l^{\otimes q}$ is isomorphic as a $G$-module to $\mu_l$ equipped with the trivial action in degrees with $q$ even, and with the Galois action in degrees with $q$ odd. Since the latter has no nontrivial fixed points, the description above follows additively. The multiplicative statement follows from the same reasoning as for \(k=\C\).\\
\item On the level of Galois cohomology, the map  induced by \(\R \rightarrow \C\) corresponds to the one induced by the map of groups which embeds the trivial group (the absolute Galois group of \C) into the Galois group of \R. Since all the Galois cohomology groups are concentrated in degree 0, the map is just the inclusion of the fixed points in $\mu_l^{\otimes q}$ under the action of the Galois group of \R~ into the fixed points of $\mu_l^{\otimes q}$ under the action of the trivial group, and the third statement follows.
\item The statement follows from \cite[Prop. 4,12]{HO} and \cite[Thm. 4.18]{HO}.
\end{enumerate}
\end{proof}

\begin{rem}
We also have \(H\Z/l_{**}=H\Z/l^{-*,-*}\). In an abuse of notation, we denote the elements in \(H\Z/l_{**}\) corresponding to \(\tau\) and \(\theta\) by the same name, where the bidegree is the same as above multiplied by -1.
\end{rem}

With the knowledge of the coefficients \(H\Z/l_{**}\) and the fact that for odd primes, they are concentrated in simplicial degree 0, we can now give a description of the dual motivic Steenrod algebra.\\
The computation of the motivic mod-$l$ Steenrod algebra over base fields of characteristic 0 is due to Voevodsky in \cite{VOE2}. The implications for the dual motivic Steenrod algebra are for example written down in the introduction of \cite{HKO}.
\begin{prop}
Let $k$ be a base field of characteristic 0, and let $l$ be an odd prime. The dual motivic Steenrod algebra \(A_{**}\) and its Hopf algebroid structure over \(k\in\{\C,\R\}\) for \(l\) an odd prime can be described as follows: \\
\[A_{**}=H\Z/l_{**}[\tau_0,\tau_1,\tau_2,...,\xi_1,\xi_2,...]/(\tau_i^2=0)\]
Here \(|\tau_i|=(2l^i-1,l^i-1)\) and \(|\xi_i|=(2l^i-2,l^i-1)\).\\ The comultiplication is given by
\[\Delta(\xi_n) = \sum_{i=0}^n \xi_{n-i}^{l^i} \otimes \xi_i\] where \(\xi_0:=1\), and
\[\Delta(\tau_n) = \tau_n\otimes 1 + \sum_{i=0}^n \xi_{n-i}^{l^i} \otimes \tau_i\]
\end{prop}

\begin{rem}
The $\tau_i$ are not related to the element $\tau$ defined above. Neither the product nor the coproduct increase the number of \(\tau_i\)'s involved in any given expression in \(A_{**}\). Hence \(A_{**}\) can be graded as an \(A_{**}\) -comodule by this number. This is similar to the classical situation at odd primes.\\
\end{rem}

\section{Generalized motivic Adams spectral sequences}
The motivic Adams spectral sequences is a computational device based on the motivic Eilenberg-MacLane spectrum. It converges to the homotopy groups of the $l$-completion \(X^{\wedge}_l\) of a motivic spectrum $X$, which we define as the Bousfield localization of $X$ at the mod-\(l\) Moore spectrum \(S/l\).

The motivic Adams Novikov spectral sequence is a similar computational device, but based on the motivic Brown-Peterson spectrum \(ABP\). It also converges to the homotopy groups of the $l$-completion.

Both are generalized Adams spectral sequences. The construction of these kind of spectral sequences is well known and mirrors the classical construction:
If \(E\) is a motivic homotopy ring spectrum such that $E_{**}E$ is flat as a (left) module over the coefficients \(E_{**}\), one can associate a flat Hopf algebroid to \(E\) (See \cite[Lemma 5.1]{NOS}). The category of comodules over this Hopf Algebroid is abelian and thus permits homological algebra (Check \cite[Appendix 1]{RAV} for the definition and basic properties of Hopf algebroids). \\

Defining \(\bar{E}\) as the cofiber of the unit map \(S\rightarrow E\) one obtains  the following cofiber sequences
\[\bar{E}^{\wedge(s+1)}\wedge X \rightarrow \bar{E}^{\wedge s}\wedge X \rightarrow E\wedge \bar{E}^{\wedge s}\wedge X\]
for any motivic spectrum \(X\) by smashing the defining cofiber sequence of \(\bar{E}\) with \(\bar{E}^s\wedge X\). These cofiber sequences give rise to so called the canonical \(E_{**}\)-Adams resolution:
\[
\xymatrix
{
...\ar[r]\ar[d]&
\bar{E}^{\wedge(s+1)} \wedge X \ar[r]\ar[d] &
\bar{E}^{\wedge s} \wedge X \ar[r]\ar[d]&
...\ar[r]\ar[d]&
\bar{E}\wedge X \ar[r]\ar[d]&
X\ar[d]\\
...&
E\wedge \bar{E}^{\wedge(s+1)} \wedge X &
E\wedge \bar{E}^{\wedge s} \wedge X &
...&
E\wedge \bar{E} \wedge X &
E \wedge X\\
}
\]
The long exact sequences of homotopy groups associated to these cofiber sequences form a trigraded exact couple and thus give rise to a trigraded spectral sequence \(E_r^{s,t,u}(E,X)\) where the differentials have the form: \[d_r: E_r^{s,t,u} \longrightarrow E_r^{s+r,t+r-1,u}\]

The long exact sequences of homotopy groups of the canonical \(E_{**}\)-Adams resolution can be identified with the (reduced) cobar complex \(C^*(E_{**}(X))\) as in the topological case by virtue of the isomorphism 
\[\pi_{**}(E\wedge E \wedge X)\cong E_{**}(E)\underset{E_{**}}{\otimes}E_{**}(X)\]
(see \cite[Lemma 5.1(i)]{NOS})
and consequently the \(E_2\)- page of the $E$-Adams spectral sequence can be described as:
\[E_2^{s,t,u}(E,X)=\text{Cotor}^{s,t,u}_{E_{**}(E)}(E_{**},E_{**}(X))\]
where \(\text{Cotor}\) denotes the derived functors of the cotensor product in the category of \(E_{**}(E)\)-comodules.

\subsection{Convergence of generalized Adams spectral sequences}
Convergence in this general situation is discussed by Bousfield in \cite{BOU}. Under certain assumptions, the spectral sequence converges completely (c.f.\cite[Chapter 6]{BOU}) to a filtration of the homotopy groups of the so called \(E\)-nilpotent completion \(X_E^{\wedge}\) of \(X\) (\cite[Chapter 5]{BOU}). In the case of the homological motivic Adams spectral sequence where \(E=H\Z/l, X=S\) and of the $l$-complete motivic Adams-Novikov spectral sequence where \(E=ABP^\wedge_l, X=S\), this turns out to be the completion at \(l\) and \(\eta\) by work of Hu, Kriz and Ormsby in \cite[Theorem 1]{HKO}. Their result tells us that for an odd prime \(l\)  we have \(S_{H\Z/l}^\wedge \cong S^{\wedge}_l\) over \(k=\C\)
and \(S_{H\Z/l}^\wedge \cong S^{\wedge}_{l, \eta}\) over \(k=\R\).

In general the \(\eta\) and \(l\)-completed homotopy groups of spheres differ from the \(l\)-completed ones. In certain bidegrees however Ormsby, Röndigs and {\O}stv{\ae}r have showed that they agree, using work of Bachmann in \cite{BM}:
\begin{prop}
	\label{OroResult}
	\cite[Theorem 1.6]{ORO}
	Let \(k=\R\) and \(l\) be an odd prime. There is an isomorphism
	\[
	\pi_{m,n}(S^\wedge_{l,\eta})\cong \pi_{m,n}(S^\wedge_{l})
	\]
	whenever the topological stable stem \(\pi_{m-n}(S^\wedge_{l})\cong (\pi_{m-n}S)^\wedge_{l}\) vanishes.
\end{prop}

We will also see later that both the motivic Adams and the motivic Adams Novikov spectral sequence converge strongly because of a vanishing line in these situations.

\begin{rem}
In \cite[7.3]{DI} Dugger and Isaksen prove that the bicompletion of the sphere spectrum \(S_{H\Z/l,ABP}\) at both the Eilenberg-MacLane and the algebraic Brown-Peterson spectrum, towards whose homotopy groups the motivic Adams-Novikov spectral sequence converges, is equivalent to the nilpotent completion at the motivic Eilenberg-MacLane spectrum alone:
\[S_{H\Z/l,ABP}^\wedge\cong S_{H\Z/l}^\wedge\]
They state their result for \(l=2\) and \(k=\C\), but the arguments they give work equally well if $l$ is an odd prime and over the base field \R. Therefore the result of Hu, Kriz and Ormsby is relevant for us also in the case of the motivic Adams Novikov spectral sequence.
\end{rem}

Specializing to the case \(k=\C\), one can consider the effect of Betti realization on the canonical Adams resolution. Betti realization preserves cofiber sequences and smash products. This implies \(R(\bar{E})=\overline{R(E)}\), and the realization of the canonical \(E_{**}\)-Adams resolution for \(X\) is the canonical \(R(E)_{*}\)-Adams resolution for the topological spectrum \(R(X)\). The induced maps on the long exact sequences of homotopy groups define a map of exact couples and therefore one of spectral sequences. Summing up, we have:
\begin{rem}
Let \(k=\C\). Then Betti realization induces a map of spectral sequences \(R_{E,X}: E_r^{s,t,u}(E,X)\rightarrow E_r^{s,t}(R(E),R(X))\)
\end{rem}
By a similar reasoning we get the following in the case \(k=\R\):
\begin{rem}
Let \(k=\R\). The functor \(\mathcal{SH}_\R\rightarrow \mathcal{SH}_\C\) associated to \(\text{Spec}(\C)\rightarrow \text{Spec}(\R)\) gives rise to a map of spectral sequences \[E_r^{s,t,u}(E,X)\rightarrow E_r^{s,t,u}(E_\C,X_\C).\]
\end{rem}

\subsection{The motivic Adams-Novikov spectral sequence}
Just as we needed the coefficients of motivic cohomology and the structure of the dual motivic Steenrod algebra for the Adams spectral sequence, we need the coeffients of the (\(l\)-completed) algebraic Brown-Peterson spectrum and the algebraic structure of the Hopf algebroid generated by it as input for the motivic Adams-Novikov spectral sequence.\\

Let \(k\subset \C\) and \(l\) be an arbitrary prime. Let \(MGL\) denote the algebraic cobordism spectrum of Voevodsky (see \cite[6.3]{VOE}) and \(MU\) the topological cobordism spectrum. The motivic Brown-Peterson spectrum \(ABP\) was first defined by Vezzosi using the Quillen idempotent in \cite{VEZ}. Hoyois proved in \cite[Remark 6.20]{HOY} that it can be defined equivalently by either the motivic Landweber exact functor theorem or as a quotient of \(MGL_{(l)}\) by any regular sequence of elements generating the vanishing ideal of the \(l\)-typical formal group law.\\
We use the same definition as \cite[6.3.1]{JOA}):
\begin{defin} Choose elements \(a^{top}_i \in MU_{2i}\) satisfying the conditions in \cite[Section 6.1]{HOY} and write \(a_i\) for their image in \(MGL_{2i,i}\) under the canonical map \(L\cong MU_{2*} \rightarrow MGL_{2*,*}\). Then \(R(a_i)=a_i^{top}\), and we can define the motivic Brown-Peterson spectrum:
	\[ABP:=MGL_{(l)}/(a_i|i\neq l^j-1)\]
	Additionally, define \(v_n^{top}:=a_{l^n-1}^{top}\) and \(v_n:=a_{l^n-1}\).
\end{defin}
We summarise a few properties of \(ABP\) that we will use later:

\begin{rem}
	\begin{enumerate}
		\item The sphere spectrum $S$ is clearly a cell spectrum of finite type over an arbitrary base field. So is the Brown-Peterson spectrum $ABP$ (regardless of the prime \(l\)) because \(MGL\) is cellular.
		\item As mentioned in \cite[6.3.2]{JOA}, we have \(R(ABP)=BP\) and \(R(v_n)=v_n^{top}\) over \(k=\C\).
		\item By \cite[Thm. 6.11]{HOY} there is an isomomorphism \(H\Z/l_{**}(ABP)\cong P_{**}\) as \(A_{**}\)-comodules, where \(P_{**}\) is the polynomial subalgebra \[P_{**}=H\Z/l_{**}[\xi_1,\xi_2,...]\] of the dual motivic Steenrod algebra.
	\end{enumerate}
\end{rem}

Both the Adams spectral sequence and the slice spectral sequence converge to the nilpotent completion of the homotopy groups of \(ABP\), so a priori the completion at the motivic Hopf map \(\eta\) might be involved. Because the motivic Hopf map \(\eta\) is not detected by \(ABP\) however, this does not matter:
\begin{prop}
	Let \(k\) be any field. Then \(ABP^\wedge_{\eta,l} \simeq ABP^\wedge_{l}\). \end{prop}
\begin{proof}
	Since \(MGL \wedge \eta=0\) (for a proof, see e.g. \cite[Lemma 7.1.1]{JOA}), we also have \(ABP \wedge \eta=0\) and \(ABP^\wedge_l \wedge \eta=0\). Hence all maps in the homotopy limit \[\underset{\leftarrow}{\text{holim}}ABP^\wedge_l/\eta^n\] are equivalences, and this homotopy limit modells the completion at $l$ and \(\eta\).
\end{proof}

\begin{prop}
	If \(k=\C\) and \(l\) an odd prime, then \[\pi_{**}(ABP_l^\wedge)={\Z_{l}}[\tau,v_1,v_2,...]\]
	If \(k=\R\) and \(l\) an odd prime, then \[\pi_{**}(ABP_l^\wedge)={\Z_{l}}[\theta,v_1,v_2,...]\]
	Here the elements \(v_i\) have bidegree \((2l^i-2,l^i-1)\).
\end{prop}
\begin{proof}
	Let \(k=\C\). Consider the motivic Adams spectral sequence for \(ABP\). Since \(ABP\) is cellular it converges to \(\pi_{**}(ABP_l^\wedge)\). The \(E_2 \)-page has the form \(E_2^{s,t,u}=\Z/l[\tau, Q_0,Q_1,..]\), where \(Q_n\) lives in degree \((1,2l^n-2,l^n-1)\). Apart from the grading by weight this agrees with the classical ASS for \(BP\), and the spectral sequence collapses at the \(E_2\)-page because there is no room for nontrivial differentials. The extension problem is solved by considering topological realization over \C.
	
\end{proof}

\begin{rem}
	More generally, the coefficients of the $l$-completed motivic cobordism spectrum \(MGL_l^\wedge\) have been computed over a $l$-low dimensional base field (a base field with exponential characteristic $p\neq l$ s.t. \(cd_l(k)\leq 2\)) by Ormsby and {\O}stv{\ae}r using the slice spectral sequence, cf. \cite[Corollary 2.6]{OO}. In particular, these computations cover the case of \(k=\R\) if \(l\) is an odd prime \(l\neq 2\).
\end{rem}

\section{The relation of the motivic Adams spectral sequence and the motivic Adams Novikov spectral sequence at odd primes}
The Adams spectral sequence and the Adams-Novikov spectral sequence at the prime 2 both provide relevant information about the homotopy groups of the spheres in the classical topological setting (e.g. \cite[Theorem 4.4.47]{RAV}) and also motivically over \(\C\) (see \cite[Chapter 8]{DI}). In fact, one can use knowledge gained in one spectral sequence to deduce differentials in the other and vice versa. In contrast to this, the classical Adams spectral sequence at odd primes yields strictly less information than the classical Adams Novikov spectral sequence. This can be stated precisely via two auxiliary spectral sequences (\cite[4.4]{RAV}), the Cartan-Eilenberg spectral sequence and the algebraic Adams-Novikov spectral sequence.
Here the Cartan-Eilenberg spectral sequence is associated to the extension of Hopf algebroids \(P_*\rightarrow A_* \rightarrow E_*\). This extension is split at odd primes, causing the Cartan-Eilenberg spectral sequence to collapse. The algebraic Adams-Novikov spectral sequence is constructed by filtering the Hopf algebroid of \(BP\) by its coaugmentation ideal. Both spectral sequences start from the same, albeit differently indexed \(E_2\)-term and converge to the \(E_2\)-terms of the Adams spectral sequence and the Adams-Novikov spectral sequence respectively. Together they form the following square:
\[
\xymatrixcolsep{2pc}
\xymatrix{
\text{Cotor}_{P_{*}}(H\Z/l_{*},\text{Cotor}^{s,t,u}_{E_{*}}(H\Z/l_{*},H\Z/l_{*})) \ar@{=>}[rr]^{AANSS}\ar@{=>}[d]^{CESS} & & Cotor_{BP^\wedge_{l,**}BP^\wedge_l}(BP^\wedge_l,BP^\wedge_l) \ar@{=>}[d]^{ANSS} \\
 Cotor_{A_**}(H\Z/l_{**},H\Z/l_{**})\ar@{=>}[rr]^{ASS} && \pi_{**}(S^\wedge_l)\\
}\]
\ \\
Motivically the situation over \(\C\) at the prime 2 is described in \cite[8.1]{AM}. It is even closer to the classical picture at odd primes for both \(k=\C\) and \(k=\R\), because the relevant Hopf algebroids are just the regraded classical Hopf algebroids base changed to the coefficients of motivic homology. As a consequence, for the purpose of the computation of the homotopy groups of the spheres, we can focus entirely on the motivic Adams-Novikov spectral sequence.\\

\subsection{The motivic Cartan Eilenberg spectral sequence}
\begin{defin}
Define \(P_{**}\) as the polynomial sub-Hopfalgebra of the dual motivic Steenrod Algebra and \(E_{**}\) as the exterior part, i.e.
\[P_{**}=H\Z/l_{**}[\xi_1,\xi_2,...]\]
\[E_{**}=H\Z/l_{**}[\tau_0,\tau_1,\tau_2,...]/(\tau_i^2=0)\]
Consider the classical extension of Hopf algebras \(P_* \rightarrow A_* \rightarrow E_*\) . We can regard all objects here as a bigraded, concentrated in degree 0 with respect to the second bidegree.\end{defin}
By our computation above \(H\Z/l_{**}\) is just a polynomial ring in one generator and hence flat over \(\Z/l\). Hence by application of \(-\underset{\Z/l}{\otimes}H\Z/l_{**}\) to the above extension, we obtain an extension of Hopf algebras \(P_{**}\rightarrow A_{**}\rightarrow E_{**}\) in the sense of \cite[A1.1.15]{RAV}, where the middle term is precisely the motivic Steenrod algebra over odd primes.

Associated to such an extension we immediately get a motivic counterpart to \cite[Theorem 4.4.3]{RAV}:

\begin{prop}
\begin{enumerate}
\item \(\text{Cotor}_{E_{**}}(H\Z/l_{**},H\Z/l_{**})\cong H\Z/l_{**}[a_0,a_1,..]\), where \(a_i \in \text{Cotor}_{E_{**}}^{1,2l^i-1,l^i-1}\) is represented by the cobar cycle \([\tau_i]\).
\item There is a motivic Cartan Eilenberg spectral sequence converging to
\[\text{Cotor}_{A_{**}}(H\Z/l_{**},H\Z/l_{**})\]
(the reindexed \(E_2\)-term of the motivic Adams spectral sequence) with the following \(E_2\)-page:
\[E_2^{s_1,s_2,t,u}=\text{Cotor}^{s_1,t,u}_{P_{**}}(H\Z/l_{**},\text{Cotor}^{s_2,*,*}_{E_{**}}(H\Z/l_{**},H\Z/l_{**}))\]
and differential \(d_r: E_r^{s_1,s_2,t,u}\rightarrow E_r^{s_1+r,s_2-r+1,t,u}\).
\item The \(P_{**}\)-coaction on \(\text{Cotor}_{E_{**}}(H_{**},H_{**})\) is given by
\[\psi(a_n)=\underset{i}{\sum}\xi_{n-i}^{l^i}\otimes a_i\]
\item The motivic Cartan Eilenberg SS collapses at the \(E_2\) page with no nontrivial extensions.
\end{enumerate}
\end{prop}

\subsection{The motivic algebraic Novikov spectral sequence}
The ideal
\[I=(l,v_1,v_2,..) \subset ABP_{l,**}^\wedge = \begin{cases}
\Z_l[\tau, v_1, v_2,...]\quad k=\C\\
\Z_l[\theta, v_1, v_2,...]\quad k=\R, l\neq 2
\end{cases}\]
is invariant in the sense of Hopf Algebroids with respect to \(ABP_{l,**}^\wedge (ABP_{l,**}^\wedge)\). Its powers define a decreasing filtration of \(ABP_{l,**}\). If we denote the class of \(l\) with \(q_0\) and the classes of \(v_i\)  with \(q_i\) for \(i\geq 1\), the associated graded of \(ABP^\wedge_{l,**}\) is:
\[E_0 ABP_{l,**}^\wedge \cong \begin{cases}\Z/l[\tau,q_0,q_1,q_2,...] \quad k=\C\\
\Z/l[\theta,q_0,q_1,q_2,...] \quad k=\R \end{cases} where\ \text{deg}(q_i) = (1,2(l^i-1))\]
By the same proof as in the classical situation in \cite[4.4.4]{RAV} this implies:
\begin{prop}
	Let \(k=\C\) and \(l\) be an arbitrary prime or let \(k=\R\) and let \(l\) be an odd prime. There is a spectral sequence called the motivic algebraic Adams-Novikov spectral sequence converging to \(Cotor_{ABP^\wedge_{l,**}ABP^\wedge_l}(ABP^\wedge_l,ABP^\wedge_l)\). The \(E_1\)-page is given by:
\[E_1=Cotor_{P_{**}}(H\Z/l_{**},E_0 ABP^\wedge_{l,**})\cong\text{Cotor}_{P_{**}}(H\Z/l_{**},\text{Cotor}^{s,t,u}_{E_{**}}(H\Z/l_{**},H\Z/l_{**}))\]
\end{prop}

Therefore we end up with essentially the same square of spectral sequences, regraded in regard to the weight degree and tensored with the coefficients of motivic homoloogy:
\[
\xymatrixcolsep{2pc}
\xymatrix{
\text{Cotor}_{P_{**}}(H\Z/l_{**},\text{Cotor}^{s,t}_{E_{**}}(H\Z/l_{**},H\Z/l_{**})) \ar@{=>}[rr]^{MAANSS}\ar@{=>}[d]^{MCESS} & & Cotor_{BP^\wedge_{l,**}BP^\wedge_l}(BP^\wedge_l,BP^\wedge_l) \ar@{=>}[d]^{MANSS} \\
 Cotor_{A_**}(H\Z/l_{**},H\Z/l_{**})\ar@{=>}[rr]^{MASS} && \pi_{**}(S^\wedge_l)\\
}
\]
Furthermore, just as in the classical topological situation the MCESS degenerates at the \(E_1\)-page.

\section{The motivic Adams-Novikov spectral sequence at odd primes}
\subsection{Elementary facts about the Adams-Novikov spectral sequence}
If \(E=ABP^\wedge_l\) and \( X=S\), the \(E\)-Adams spectral sequence is called the \(l\)-completed motivic Adams Novikov spectral sequence (from now on abbreviated to MANSS) and was studied for example in \cite[3]{OO},\cite{HKO2}. For \(l\)-low dimensional fields Ormsby and {\O}stv{\ae}r have described the \(E_2\)-page of the \(l\)-primary MANSS in terms of the \(E_2\)-page of the classical ANSS and of the coefficients \(H\Z_{l**}\) in \cite{OO}, in order to calculate \(\pi_{1,*}\) over these fields. We want to use their results. First we need to compute \(ABP^\wedge_{l,**}(ABP_{l,**}^ \wedge)\), which is known by a result of Naumann, Spitzweck and Østvær (\cite[Lemma 9.1]{NOS}), as well as the Hopf algebroid structure induced by \(ABP_l^\wedge\) :
\[ABP^\wedge_{l,**}(ABP^\wedge_l)\cong ABP^\wedge_{l,**}\underset{\pi_{*}(BP)}{\otimes}BP_*(BP)\]
Over \C~ and \R~ the Hopf algebroid structure is given by tensoring the Hopf algebroid of the Brown Peterson spectrum \(BP\) with \(\Z_l[\tau]\) resp. \(\Z_l[\theta]\) over \(\Z_l\).
\begin{rem}
\begin{enumerate}
\item The Hopf algebroid induced by \(ABP_l^\wedge\) is flat for \(k\in\{\R, \C\}\). 
\item Let \(k=\C\) and \(l\) be any prime. Then \(ABP_{H\Z/l}^{\wedge}\cong ABP^{\wedge}_l\) by \cite[Theorem 1]{HKO}, so this spectral sequence agrees with the \(H\Z/l\)-completed MANSS considered in \cite{DI}.
\item Let \(k=\R\) and \(l\) be an odd prime \(l\neq 2\).\\
By comment 1 on \cite[Theorem 1]{HKO} , the $H\Z/l$-nilpotent completion of the sphere spectrum \(S\) is a completion at $l$ and \(\eta\), and this completion is not equivalent to the completion at $l$ alone.
\end{enumerate}
\end{rem}

We recall some facts about the topological ANSS built from the spectrum \(BP\), in particular sparseness, which we are going to need in order to compare the classical ANSS to the MANSS. Because the structure of the classical and the motivic Adams-Novikov spectral sequence are very similar, we can then use these facts to prove similar statements for the MANSS.
\begin{rem}
	\begin{enumerate}
		\label{ANSSfacts}
		\item \cite[5.1.23 (a)]{RAV} There is a vanishing line of slope 1 for the \(E_2\)-term of the classical ANSS in Adams grading: \[Cotor^{s,s+s'}_{BP_*BP}(BP_*,BP_*)=0\quad if\ s'<s\]
		\item \(E_2^{0,0}=Z_{(l)}\) and \(E_2^{0,t}=0\) for \(t\neq 0\). (cf. \cite[5.2.1 (b)]{RAV})
		\item If \((s,t)\neq(0,0)\), then \(Cotor^{s,t}_{BP_*BP}(BP_*,BP_*)\) is a finite \(l\)-group. (cf. \cite[5.2.1 (a)]{RAV} together with the previous statement). In particular, if we construct the generalized Adams spectral sequence for \(BP_l^\wedge\), then this spectral sequence agrees with the ANSS in all bidegrees away from \((0,0)\), and in this bidegree there is a single copy of \(\Z_l\).
		\item \(E_2^{s,t}\neq 0\) only if \(t\) is divisible by \(q=2l-2\). This means that a differential  \(d_r\) can be nontrivial only if \(r\equiv 1(\mod q)\). Hence \(E_{nq+2}\cong E_{nq+3}\cong ... \cong E_{(n+1)q+1}\).\cite[4.4.2]{RAV}
	\end{enumerate}
\end{rem}

The following results take the same form as in the case \(l=2\), discussed in \cite{Isa} and \cite{HKO2}:
\begin{prop}
\label{MANSSfacts}
Assume that we know the \(E_2\)-page of the classical ANSS (associated to either \(BP\) or \(BP_l^\wedge\)) in a certain range. The \(E_2\)-page of the odd-primary MANSS over \C~ can then be constructed from this information as follows:
\begin{enumerate}
\label{E2PageDescription}
\item \(E_2^{0,0,u}=0\) if \(u>0\) and \(E_2^{0,0,u}\cong \Z_l\tau^u\) if \(u \leq 0 \). By the multiplicative structure of the MANSS, each \(E_2^{s,t,*}\) is a \(\Z_l[\tau]\)-module, so we can speak of \(\tau\)-torsion.
\item Let \( (s,t) \neq (0,0)\): For each group \(\Z/l^n\) in \(E_2^{s,t}\) of the ANSS, there is a group \(\Z/l^n[\tau]\) in \(E_2^{s,t}\) of the MANSS, and its generator as a \(\Z_l[\tau]\)-module has weight \(\frac{t}{2}\). There are no other groups in \(E_2^{s,t,u}\) of the MANSS.
\item The vanishing line in the classical ANSS carries over to the MANSS, so the MANSS converges strongly.
\label{ANSSVL}
\item \(E_2^{s,t,u}\neq 0\) only if \(t\) is divisible by \(q=2l-2\). This means that a differential  \(d_r\) can be nontrivial only if \(r\equiv 1 (\mod q)\). Hence \(E_{nq+2}\cong E_{nq+3}\cong ... \cong E_{(n+1)q+1}\).
\item The MANSS over \R\ maps injectively to the MANSS over \C.
\end{enumerate}
\end{prop}
\begin{proof}
For \(k\in\{\C,\R\}\) and \(l\) an odd prime we can use the computations of \(\pi_{**}(ABP^\wedge_l)\)  to determine the cobar complex of the Hopf algebroid associated with \(ABP_{l}^\wedge\):

\[C^*(ABP^\wedge_l)\cong C^*(BP)\underset{\pi_*(BP)}{\otimes}\pi_{**}(ABP^\wedge_l)\cong \begin{cases} C^*(BP)\underset{\Z_{l}}{\otimes}\Z_{l}[\tau] \quad |\ k=\C\\
C^*(BP)\underset{\Z_{l}}{\otimes}\Z_{l}[\theta] \quad|\ k=\R
\end{cases}\]
The map \( C^*(ABP^\wedge_{l,\R}) \rightarrow C^*(ABP^\wedge_{l, \C})\) maps \(\theta \mapsto \tau^2\).

Application of the universal coefficient theorem to the cobar complex then yields the following short exact sequences which relate the \(E_2\)-term \(E_2^{s,t}\) of the $l$-completed MANSS to the \(E_2\)-page \(E^{s,t}_{2,top}\) of the classical Adams-Novikov spectral:
\begin{itemize}
\item \(k=\C:\quad
0 \rightarrow E_{2,top}^{s,t}\underset{\Z_{l}}{\otimes}\Z_{l}\tau^n \rightarrow E_2^{s,t,\frac{t}{2}-n}\rightarrow \text{Tor}_1^{\Z_{l}}(E_{2,top}^{s-1,t},\Z_{l}\tau^n)\rightarrow 0
\)
\item \(k=\R:\quad 0\rightarrow E_{2,top}^{s,t}\underset{\Z_l}{\otimes}\Z_{l}\theta^n \rightarrow E_2^{s,t,\frac{t}{2}-n}\rightarrow \text{Tor}_1^{\Z_{l}}(E_{2,top}^{s-1,t},\Z_{l}\theta^n)\rightarrow 0\)
\end{itemize}
Note that \(E_{2,top}^{s,t}\) vanishes for odd \(t\) because of sparseness.

The torsion term vanishes in both cases because the second argument is free over the ground ring. This proves the first two statements. The third and fourth statement can either be proven directly by examining the cobar complex or they can be seen as a corollary of the first two. The last point follows from the fact that the map of exact couples which induces the map of spectral sequences is an isomorphism in each weight where the MANSS over \R\ does not vanish.
\end{proof}

\begin{rem}
For \(l\neq 2\) and \(k \in \{\C, \R \}\) the dual motivic Steenrod algebra can also be expressed in terms of the classical dual Steenrod algebra and the motivic homology of a point, so one can use the same arguments to show that the \(E_2\)-term of the motivic Adams spectral sequence is given by the \(E_2\)-term of the classical Adams spectral sequence tensored with \(\Z/l[\tau]\) over \C~ and \(\Z/l[\theta]\) over \R. The weights are dictated by the cobar representatives of the generators.
\end{rem}

\subsection{The structure of the motivic Adams-Novikov spectral sequence at odd primes}
The topological realization functor \(R: \mathcal{SH}_\C \rightarrow \mathcal{SH}_{top}\) induces a map of towers and hence a map of spectral sequences \(\Psi\) from the $l$-completed MANSS to the $l$-completed topological ANSS, which is just the classical ANSS away from the bidegree \((0,0)\). We want to use the arguments of the case \(l=2\) (\cite[Lemma 16]{HKO2}) to show that there can be no differentials in the motivic Adams-Novikov spectral sequence that vanish under \(\Psi\). Note that this implies that the \(\tau\)-primary torsion parts of the MANSS cannot support nontrivial differentials, because they map to zero under the map \(\Psi\).\\
For this argument we need the fact that all \(\tau\)-primary torsion vanishes in weight \(0\):

\begin{lemma}
Let \(l\) be an odd prime and \(k=\C\). 
\begin{enumerate}[ref=\thetheorem.\arabic*]
\item	\label{kernelIsTorsion}
The restriction of the map \(\Psi\) to weight 0 is an isomorphism of spectral sequences (where the ANSS is considered a trigraded spectral sequence concentrated in degree \(0\) with respect to the third grading), and any element in the kernel of \(\Psi\) must be \(\tau\)-torsion.
\item \label{diffImp}
\(d_{r}(x)=0 \implies d'_{r}(\Psi(x))=0\).
\item The element \(\tau \in E_2^{0,0,-1}\) is a permanent cycle in the MANSS and defines an element \(\tau \in \pi_{0,-1}(S^\wedge_l)\).
\end{enumerate}
\end{lemma}
\begin{proof}
\begin{enumerate}
\item In weight \(0\) topological realization induces an isomorphism already on the level of exact couples, so \(\Psi\) is an isomorphism of spectral sequences in weight \(0\) for all \(r\).\\
Assume \(\Psi(x)=0\) for an element \(x\) in weight \(u\). \(\Psi\) respects the product structure, so \(\Psi(\tau^u \cdot x)=\Psi(\tau^u)\cdot \Psi(x)=1\cdot \Psi(x)=0\). Because \(\Psi\) is an isomorphism in weight \(0\), it follows that \(\tau^u\cdot x=0\).
\item \(\Psi\) is a map of spectral sequences and commutes with differentials.
\item It is clear by the tridegree of \(\tau\) in the MANSS that it is a permanent cycle.
\end{enumerate}
\end{proof}

We can now prove that the motivic Adams-Novikov spectral sequence can be completely described in terms of the topological Adams-Novikov spectral sequence (compare \cite[Lemma 16]{HKO2} for the corresponding statement if \(l=2\)):
\begin{prop}
\label{PropChapter2}
Let \(l\) be an odd prime and \(k=\C\). 
\begin{enumerate}[ref=\thetheorem.\arabic*]
\item The \(E_r\)-page of the $l$-completed MANSS \(E^{s,t,*}_r\) contains the following elements (and only these) in each bidegree \((s,t)\):\label{ManssDescription}
\begin{itemize}
\item The subgroup of elements which are not \(\tau\)-primary torsion is given by \(E_{r,top}^{s,t}\underset{\Z_l}\otimes \Z_l[\tau]\), where the elements of \(E_{r,top}^{s,t}\) are considered to have weight \(\frac{t}{2}\).
\item Let \(q=2l-2\). For every nontrivial classical differential \[d_r': E_{r',top}^{s,t} \rightarrow E_{r',top}^{s+r',t+r'-1}\] with \(r'<r\) there is a \(\tau\)-primary torsion subgroup of order \(\frac{r'-1}{2}\). This group is generated by an element in \(E_r^{s+r',t+r'-1,\frac{t+r'-1}{2}}\) which neither supports nor receives any nontrival differential. (Note that \(r'=nq+1\) because of sparseness, so \(r'\) must be odd).
\end{itemize}
In particular, the \(\tau\)-primary torsion subgroups at \(r=nq+1\) have order of at most \(\frac{r-1}{2}\).
\item For any \( x \in E_{r}\), we have \(d'_{r}(\Psi(x))=0 \implies d_{r}(x)=0\). As noted before, this implies that a \(\tau\)-primary torsion element cannot support a nontrivial differential.
\item For each \(E^{s,t}_\infty\) of the classical ANSS, there is a group \(E^{s,t,\frac{t}{2}}_\infty[\tau]\) in the \(E_\infty\)-page of the MANSS. The \(\tau\)-primary torsion subgroup of the \(E_\infty\)-page is generated by the groups above.
\end{enumerate} 
\end{prop}
\begin{proof}
We prove the first two claims by induction over \(r\). The third claim then follows from the other two.\\
The first claim is true for \(r=2\) by \ref{E2PageDescription} and  the second one is by \ref{kernelIsTorsion}. In particular there is no \(\tau\)-primary torsion on the \(E_2\)-page.\\
Because the \(E_2\)-page is sparse, all the differentials are zero for degree reasons if \(r\neq nq+1\), and the induction step is trivial. We can therefore restrict to the case \(r=nq+1\). Assume all the claims are true for \(r=nq+1\), and we wish to show that they also hold for \(r+1\).\\
Let \[E_{r,top}^{s,t}\ni x  \overset{d'_r}\longrightarrow y \in E_{r,top}^{s+r,t+r-1}\] be a nontrivial differential in the topological Adams Novikov spectral sequence. By \ref{diffImp} there is a nontrivial differential \[E_{r}^{s,t,0}\ni \Psi^{-1}(x)  \overset{d_r}\longrightarrow \Psi^{-1}(y) \in E_{r}^{s+r,t+r-1,0}\] in the motivic Adams Novikov spectral sequence for unique elements \(\Psi^{-1}(x)\)  and \(\Psi^{-1}(y)\) in weight 0. By the first part of the induction assumption there are elements \(\tilde{x} \in E_{r}^{s,t,\frac{t}{2}}\) and \(\tilde{y} \in E_{r}^{s+r,t+r-1,\frac{t+r-1}{2}}\) s.t. \(\tau^{\frac{t}{2}}\tilde{x} = \Psi^{-1}(x)\) and \(\tau^{\frac{t+r-1}{2}}\tilde{y}=\Psi^{-1}(y)\), where \(t+r-1=t+nq\) is divisible by two since both \(t\) and \(q\) are. Because \(\Psi(\tau)=1\) and because there can be no \(\tau\)-primary torsion in the target at weight \(\frac{t}{2}\) by the first part of the induction assumption, this implies \(d_r(\tilde{x}) = \tau^{\frac{r-1}{2}}\tilde{y}\). Hence \(\tilde{y}\) generates a \(\tau\)-primary torsion group of order \(\frac{r-1}{2}\) on the \(E_{r+1}\)-page of the MANSS. This proves the \(\tau\)-primary part of the description of the \(E_{r+1}\)-page. The rest of the description then follows from the second part of the induction assumption.\\
To prove the second part of the induction, consider a nontrivial differential in the heighest possible weight \(d_r: E^{s,t,\frac{t}{2}} \rightarrow E^{s+r,t+r-1,\frac{t}{2}}\) in the MANSS. By the induction assumption there is no \(\tau\)-primary torsion in this or in lower weights, so multiplication by \(\tau^\frac{t}{2}\) is an isomorphism between weight \(\frac{t}{2}\) and weight 0. By \ref{kernelIsTorsion} \(\Psi\) is an isomorphism in weight 0, which implies the statement of the second part.
\end{proof}

\begin{prop}
\label{PropChapter2addendum}
Consider k=\R\ and \(l\neq 2 \). The results of the previous proposition apply if we replace \(\tau\) with \(\theta\).  
\end{prop}
\begin{proof}
By point 5 of \ref{MANSSfacts}, the map of MANSSs induced by \(\R \rightarrow \C\) is the inclusion of the terms in even weight on each page.
\end{proof}
Now let \(k=\C\) as before. It is known by a result of Morel (\cite{Mor2}) that \(\pi_{s',u}(S)=0\) if \(u>s'\). For the \(l\)-completed sphere this can also be deduced from the vanishing line of the ANSS:

\begin{prop}
Let k = \C\  and \(l\neq 2\). Then \(\pi_{s',u}(S^\wedge_l)=0\) if \(u>s'\).
\end{prop}
\begin{proof}
If we index the MANSS according to the usual Adams grading \((s,s'):=(s,t-s)\), the column \(s'\) converges to \(\pi_{s',*}(S^\wedge_l)\). 
By \ref{MANSSfacts} the entry \(E^{s,t,u}_2\) of the \(E_2\)-page of the MANSS is zero if \(t<2s\) or equivalently \(s'<s\). Because the weight of a generator in \(E^{s,s',*}_2\) is \(\frac{s'+s}{2}\) (and the weight of arbitrary elements is lesser or equal to this) and because the weight grows with \(s\), a nonzero element in the column \(s'\) can at most have weight \(s'\), if the element lies in the maximal nonzero filtration degree \(s=s'\).
\end{proof}

\begin{prop}
\label{tauIsoProp}
Let k = \C\  and \(l\neq 2\). On \(\pi_{s',u}(S^\wedge_l)\) multiplication by \(\tau\) is an isomorphism if \(u\leq \frac{s'}{2}+1\) and \(s'>0\), or if \(s'=0\) and \(u\leq 0\). Since the groups \(\pi_{s',0}(S^\wedge_l)\) in weight 0 are known to be the (completed) classical groups by \cite{LEV} this determines the motivic groups in the given range.
\end{prop}
\begin{proof}
	In a given bidegree \((s, t)\) the lowest weight where \(\tau\)-primary torsion can occur is \(\frac{t-r+1}{2}+1\) by \ref{ManssDescription}, if there is a nontrivial differential \(d_r\) with source in bidegree \((s-r, t-r+1)\). If we index the MANSS according to the usual Adams grading \((s,s'):=(s,t-s)\), this weight is equal to \(\frac{s'+(s-r)+1}{2}+1\), and \(s-r\) corresponds to the line where the source of the differential is.
	
	There are no elements in the 0-line for \(s>0\). If the differential has source in the 1-line (\(s-r=1\)) the lowest absolute weight where \(\tau\)-primary torsion can appear is \(\frac{s'+2}{2}+1=\frac{s'}{2}+2\), and for higher lines, it is larger. Therefore there cannot be \(\tau\)-primary torsion in weight lower or equal to \(\frac{s'}{2}+1\) for fixed \(s'\).
\end{proof}

\begin{rem}
Because the map from the MANSS over \R\ to the MANSS over \C\ is just the inclusion of	the even weighted parts of the latter, the same statements hold over \R\ if we replace \(\tau\) with \(\theta\).
\end{rem}

\subsection{Examples of differentials and elements in the motivic Adams Novikov spectral sequence}

Let k=\C\ and let \(l\neq 2\) be any odd prime. The first nontrivial differential in the topological Adams Novikov spectral sequence (cf. \cite[Thm. 4.4.22]{RAV}) is \[d_{2l-1}(\beta_{l/l})=a\alpha_1\beta_1^l\] where \(a\) is  an unknown constant. In the classical Adams grading \(s,t-s,u\) the motivic analogues of these elements live in the following degrees and weights:
\begin{table}[h]
\centering
\begin{tabular}{ l | c | c | c }
	Name & s & t-s & u\\
	\hline
	\(\alpha_1\) & \(1\) & \(2l-3\) & \(l-1\) \\
	\(\beta_1\) & \(2\) & \(2l^2-2l-2\) & \(l^2-l\) \\
	\(\beta_{l/l}\) & \(2\) & \(2l^3-2l^2-2\) & \(l^3-l^2\)\\
\end{tabular}
\end{table}
Application of \ref{PropChapter2}  then yields the following example of a \(\tau\)-torsion group:
\begin{example}
	\label{firstExample}
In the stem \(t-s = 2l^3-2l^2-3\) there is a $\tau$-torsion group of order $l-1$. If we insert for instance \(l=3\) this yields a permanent cycle in the MANSS in degree \(s=7, t-s=33, u=20\) that is annihilated by \(\tau^2\).	
\end{example}

\begin{rem}Any classical differential of the form \(\Z/l^2 \overset{\cdot l}{\rightarrow}\Z/l^2\) in the ANSS would imply an entry of the form \(\Z/l[\tau]/(l\tau^n)\) in the MANSS. However, we do not know an example of such a differential.
\end{rem}

Consider the classical ANSS and the elements representing the Hopf maps. The motivic counterpart of the Hopf maps have been defined in \cite{DI3}. For \(l=3\), the element \(\alpha_1\) lives in bidegree \(s=1, t-s=3\) and is not involved in any differentials. It represents the Hopf map $\nu$. Its motivic counterpart has the same bidegree \((s,t-s)\) and weight 2. Similarly, the element \(\alpha_2\) has bidegree \(s=1, t-s=7\). It is a permanent cycle and represents the Hopf map \(\sigma\). Its motivic counterpart has weight 4.\\
For $l=5$ the element \(\alpha_1\) in the degree \(s=1, t-s=7\) similarly is a permanent cycle, representing the Hopf map \(\sigma\). Its motivic counterpart also has weight 4. At odd primes, products of the \(\alpha\)-elements all vanish in the classic Adams Novikov spectral sequence. Because there is no \(\tau\) on the \(E_2\) page of the motivic Adams Novikov spectral sequence, this is true motivically as well. We can conclude:
\begin{rem}
	All three of this elements are nilpotent in the MANSS.
\end{rem}

Let us now consider the real case \(k=\R\) and once again let \(l\neq 2\) be any odd prime. The element \(\alpha_1\) of the MANSS is an infinite cycle. It's classical counterpart generates the \(l\)-torsion group in the classical stable stem \(2l-3\), the lowest stem in which \(l\)-torsion occurs. Since this element lives in positive weight the result \cite[Theorem 1.6]{ORO}, cited as \ref{OroResult}, of Ormsby, Röndigs and {\O}stv{\ae}r  applies. This means that \(\alpha_1\) actually represents an element in the \(l\)-completed stable motivic stems, not just in the \(l\)- and \(\eta\)-completed ones, and generates \(\pi_{2l-3,2}(S_l^{\wedge})\cong\Z/l\).

The element \(\theta\) is a permanent cycle because of its position in the spectral sequence, just like \(\tau\) is over \C. If we multiply \(\alpha_1\) by powers of \(\theta\) we still get a permanent cycle. The product \(\theta\alpha_1\) has weight 0. In this weight the \(l\)-completed topological stable stems are generated by \(\alpha_1^{top}\) and in particular do not vanish anymore.
In this case the proof of of Ormsby, Röndigs and {\O}stv{\ae}rs result tells us that \(\pi_{2l-3,0}(S_l^\wedge[\tfrac{1}{\eta}])\cong \pi_{m-n}(S^\wedge_l)\cong \Z/l\). The isomorphism \[\pi_{2l-3,0}(S_l^{\wedge})\cong \pi_{2l-3,0}(S_{l, \eta}^\wedge)\oplus \pi_{2l-3,0}(S_l^\wedge[\tfrac{1}{\eta}])\] therefore implies:
\[\pi_{2l-3,0}(S_l^{\wedge})\cong\Z/l \oplus \Z/l.\]

\ \\
\textit{Sven-Torben Stahn\\
	Fachgruppe Mathematik und Informatik\\
	Bergische Universität Wuppertal\\}
SvenTorbenStahn@gmail.com

\end{document}